\documentclass[11pt]{article}

\usepackage[utf8]{inputenc}
\usepackage{color}
\usepackage{subfigure}
\usepackage{url}
\usepackage{graphicx}
\usepackage{amsmath}
\usepackage{amsthm}
\usepackage{bm}
\usepackage[plainpages=false]{hyperref}
\usepackage{enumerate}

\newcommand*{\doi}[1]{doi: \href{https://dx.doi.org/#1}{\urlstyle{rm}\nolinkurl{#1}}}
\newcommand*{\arxiv}[1]{arXiv:  \href{https://arxiv.org/abs/#1}{\urlstyle{rm}\nolinkurl{#1}}}

\usepackage{titlesec}
\titleformat{\section}
	{\normalfont\Large\bfseries\filcenter}{\thesection.}{1 ex}{}
\titleformat{\subsection}
	{\normalfont\normalsize\bfseries}{\thesubsection.}{1 ex}{}
\titleformat{\subsubsection}[runin]
	{\normalfont\normalsize\bfseries\filcenter}{\thesubsection.}{1 ex}{}

\usepackage[charter]{mathdesign}
\usepackage[mathcal]{eucal}

\usepackage[margin=1.2 in]{geometry}

\usepackage[square,numbers,sort&compress]{natbib}

\usepackage{thmtools,thm-restate}
\declaretheorem[within=section]{theorem}
\declaretheorem[sibling=theorem]{lemma}

\declaretheorem[sibling=theorem]{corollary}

\declaretheorem[style=remark,sibling=theorem,qed={$\diamondsuit$}]{remark}
\declaretheorem[style=definition,sibling=theorem]{definition}

\declaretheorem[style=remark,sibling=theorem]{example}

\renewcommand{\vec}[1]{\mathbf{#1}}

\newcommand\cE{\mathcal{E}}

\newcommand\EE{\mathbb{E}}

\newcommand\NN{\mathbb{N}}
\newcommand\QQ{\mathbb{Q}}
\newcommand\RR{\mathbb{R}}

\newcommand{\Euc}{{\rm Euc}}

\DeclareMathOperator{\rank}{rank}

\newcommand{\defn}[1]{\emph{\color{blue} #1}} 

\newcommand{\ones}{\mathbb{1}}
\newcommand{\eps}{\varepsilon}

\usepackage{tikz}
\usetikzlibrary{arrows}

\begin{document}
\title{Frameworks with coordinated edge motions}
\author{Bernd Schulze \thanks{
Department of Mathematics and Statistics, Lancaster University, 
Lancaster, UK. \url{b.schulze@lancaster.ac.uk}}
\and
Hattie Serocold \thanks{
\url{hattie.serocold@gmail.com}
}
\and Louis Theran
\thanks{School of Mathematics and Statistics, University of St Andrews, St Andrews, Scotland. \url{lst6@st-and.ac.uk}}}
\date{}

\maketitle

\begin{abstract}
We develop a rigidity theory for bar-joint 
frameworks in Euclidean $d$-space in which specified classes of edges 
are allowed to change length in a coordinated fashion that requires
differences of lengths to be preserved within each class.  Rigidity 
for these coordinated frameworks is a generic property, and 
we characterize the rigid graphs in terms of redundant
rigidity in the standard $d$-dimensional rigidity matroid.  We
also interpret our main results in terms of matroid unions.
\end{abstract}


\section{Introduction}\label{sec:introduction}
A (bar-joint) framework $(G,p)$ is a graph $G=(V,E)$ and a 
map $p:V\to \EE^d$. By identifying $p$ with a vector in $\EE^{d|V|}$ (using any order on $V$), we may also refer to $p$ as a configuration of $|V|$ points in $\EE^d$.
Intuitively, we may think of a  framework as a collection of
fixed-length bars (corresponding to the edges of $G$) which are connected at their ends
by joints (corresponding to the vertices of $G$) that allow bending in any direction of $\EE^d$.

One fundamental question in rigidity theory is whether all edge-length preserving, continuous  
motions of a given framework are rigid  body motions.  In this case, a framework is called
\defn{rigid} and otherwise \defn{flexible}.  (See Section 
\ref{sec:background} for formal definitions.)

\paragraph{Generic rigidity}
In general, whether a framework $(G,p)$ is rigid or flexible depends on both $G$ and $p$; that is, rigidity 
is a geometric property.  However, there is a typical behavior as $p$ varies and $G$ is fixed.
\begin{definition}\label{def: generic}
A point configuration $p$ of $n$ points in dimension $d$ 
is \defn{generic} if the coordinates of the points $p(i)$
do not satisify any polynomial equation with coefficients in $\QQ$.
Generic points are dense in the space of $n$-point configurations.

If $G$ is a graph with $n$ vertices, a framework $(G,p)$ is called 
generic if $p$ is generic.
\end{definition}
Asimow and Roth proved the following fundamental result \cite{AR78}.
\begin{theorem}\label{thm: ar}
For every dimension $d$ and graph $G$, either 
every generic framework $(G,p)$ in $\EE^d$ is rigid or every generic framework $(G,p)$
in $\EE^d$ is flexible.
\end{theorem}
A consequence is that we may investigate the generic rigidity 
properties of a graph $G$.
\begin{definition}\label{def: rigid graphs}
Let $d\in \NN$ be a dimension and $G$ a graph. Then $G$ is 
\defn{generically rigid} in dimension $d$ if every generic framework $(G,p)$
is rigid; otherwise every generic $(G,p)$ is flexible and $G$ is \defn{generically flexible}
in dimension $d$.

If $G$ is generically rigid in dimension $d$, but no proper spanning subgraph of $G$ 
is generically rigid, then $G$ is \defn{isostatic} in dimension $d$.
\end{definition}

\paragraph{Rigidity matroids}
It is implicit in \cite{AR78} and first explicitly observed and used by 
Lovász and Yemini in \cite{LY82} that generic rigidity has a matroidal structure.
\begin{definition}\label{def: rigidity matroid}  
Fix a dimension $d$ and let $n\ge d$.  Let $E^n$ be the edges of the complete graph $K_n$.
The matroid  on the ground set $E^n$ of rank $dn - \binom{d+1}{2}$
that has as its bases the edge sets of the isostatic graphs with $n$ vertices in dimension $d$ is called the $d$-dimensional \defn{rigidity matroid} of $K_n$ and is denoted by $M_{d,n}$.

The restriction of $M_{d,n}$ to the edges of an $n$ vertex graph $G$ is 
the \defn{rigidity matroid of $G$}, $M_d(G)$.
\end{definition}
Lovász and Yemini define $M_{d,n}$ in terms of a linearization of rigidity called
infinitesimal rigidity that we discuss in more detail in Section \ref{sec:background}.
This approach is now standard in the field; see, e.g., Whiteley's survey \cite{Wh96} 
for an overview on the interplay between matroid theory and rigidity problems.  

It is easy to check, using a randomized algorithm based on Gaussian elimination, 
whether a specific graph $G$ is generically rigid in dimension $d$ for any $d$ and number of 
vertices $n$ (a detailed analysis is in \cite{GHT}, but this is a folklore fact).
On the other hand, except for dimensions $d=1$, which is folklore, and $d=2$, 
which is due to Pollaczek-Geiringer \cite{hilda} (and later rediscovered by 
Laman \cite{L70}), a combinatorial characterization of the matroids $M_{d,n}$
is a notable open problem \cite[Sec. 61.1.2, ``Open problems'']{HoDCG}.

\paragraph{Coordinated rigidity: motivation and results}

In recent work, Nixon, Schulze, Tanigawa and Whiteley \cite{NSTW15} defined
a generalization of frameworks that enlarges the class of allowed motions.  The vertices of $(G,p)$ are partitioned into $k+1$ different classes, 
$V_0, V_1, \ldots, V_k$.  The set of allowed configurations $p$ is constrained
so that for $j\ge 1$ all vertices $i\in V_j$ have the same distance to the origin (but this distance may change), 
 and all the vertices $i$ in $V_0$ lie on the unit sphere.  The
allowed motions are the continuous deformations in the space of allowed
configurations. The motivation for studying these types of frameworks is to interpolate between rigidity in dimension $d$ and dimension $d+1$.

A model for this expanding spheres setup, which is present in \cite{NSTW15}, is based on Whiteley’s
coning construction \cite{coning}. We first add a new vertex $v_0$ to $G$ and fix it at the origin (the centre of the spheres) and then we connect $v_0$ to each of the 
vertices of $G$. The new edges joining $v_0$ with the vertices in $V_0$ must have fixed unit length, and the remaining new edges joining $v_0$ with vertices in $V_1\cup \ldots \cup V_k$ do not have fixed length, but all of the ones in the same class $V_j$ must have the same length.

Inspired by \cite{NSTW15} we consider frameworks in which  not all of the bars are fixed-length in a more general 
fashion.  We identify, in advance $k\in \NN$ 
``coordination classes'' of edges which are allowed to change 
their length, subject to edge length differences being 
preserved within each coordination class.

Our study of coordinated rigidity
in such a general setup is also motivated by some recent results in condensed matter theory. In 
\cite{EHKTvH15,RPW+18,HLN18} it is shown that (nearly) minimally 
rigid frameworks can be 
``tuned'' to have a number of interesting 
geometric and material properties.  The results we
present here could potentially form the combinatorial 
part of a design methodology for 
such ``meta-materials''. Other potential practical applications arise from the analysis of frameworks modelling engineering structures that are driven by a collection of pistons which are all connected to a central pump, and so will extend or contract based on the pressure across the whole system, or structures whose members are made of multiple different types of materials, which may expand at different rates when the structure is heated.

In this paper we develop the 
continuous and infinitesimal rigidity theory for coordinated
frameworks and show that rigidity of coordinated frameworks
is a generic property.  
Our main result, Theorem \ref{thm:all-k}, 
shows that generic rigidity of a coordinated framework has a 
characterization in terms of the generic rigidity properties of 
a bar-joint framework with the same underlying graph. 
As a corollary, we identify the $d$-dimensional coordinated 
rigidity matroid for a fixed partition of the edges of the graph into coordination 
classes with a matroid union of the standard rigidity matroid  and an 
associated partition matroid. In particular, this provides a deterministic polynomial time algorithm for checking generic coordinated rigidity in dimension $2$ for any $k$.

\paragraph{Organization}

The structure of the paper is as follows. We start by briefly introducing the necessary definitions and results from standard (finite and infinitesimal) rigidity theory in Section \ref{sec:background}. These definitions and results are then adapted in Section \ref{sec:coordinated-frameworks} to the coordinated context described above.  Characterizations for generic coordinated rigidity in arbitrary dimension are then given in Section \ref{sec:bridges}. Finally, in Section~\ref{sec:closing} we discuss algorithms for checking generic coordinated rigidity and outline some further directions of research.

\section{Rigidity background}\label{sec:background}
We start by introducing the definitions, notation, and basic results from rigidity theory that are
required for the development of a rigidity theory for coordinated frameworks. (See \cite{HoDCG}, for example, for further details.)

\subsection{Graphs} 
We denote graphs by $G=(V,E)$, where $V$ is 
the set of vertices and $E$ is the set of edges. 
In cases where $G$ is not clear from the context, we  write $V(G)$ for $V$ and $E(G)$ for $E$.  We usually
use $n$ and $m$ to denote the number of vertices and edges,
respectively, and write edges as unordered pairs $\{i,j\}$ of 
vertices.  We also use the notation $e$ for an edge when 
the endpoints aren't important.

\subsection{Point configurations} 
Fix a dimension $d$.
A \defn{$d$-dimensional configuration} $p$ is an ordered
tuple of $n$ points $(p(1),  \ldots, p(n))$
in $\EE^d$.  Let $\Euc(d)$ be the group of \defn{rigid motions} of 
$\EE^d$.  We define configurations $p$ and $q$ to be 
\defn{congruent} if there is a $T\in \Euc(d)$ such that 
$q(i) = T(p(i))$ for all $1\le i\le n$.

Fixing an affine structure 
and an origin, we can identify points in $\EE^d$ with 
their coordinates in $\RR^d$,
so we may regard $p$ 
as a mapping $[n]\to \RR^d$ or a vector in 
$\left(\RR^d\right)^n\cong \RR^{dn}$.

The tangent space to $\EE^d$ is $\RR^d$ at every point, 
so we have an identification between \defn{velocity fields} $p'$
supported on $p$ and $\left(\RR^d\right)^n\cong \RR^{dn}$.

\subsection{Frameworks, rigidity, and flexibility}\label{sec:fwrig}
We now define bar-joint frameworks formally.
\begin{definition}\label{def: framework}
A \defn{$d$-dimensional (bar-joint) framework} $(G,p)$ is defined by a graph $G$ with $n$ vertices and an $n$-point 
configuration $p$ which assigns a point in $\RR^d$ to each vertex of $G$.
\end{definition}

\paragraph{Continuous rigidity}
Underlying the notion of rigidity are the concepts of equivalence and
congruence.
\begin{definition}\label{def: eqv-cong}
Two $d$-dimensional frameworks $(G,p)$ and $(G,q)$ are \defn{equivalent} if 
\begin{eqnarray}\label{eq:edge-eqv}
	||p(j) - p(i)|| = ||q(j) - q(i)|| & 
    \text{for all $\{i,j\}\in E$}
\end{eqnarray}
Frameworks $(G,p)$ and $(G,q)$ are \defn{congruent}
if $p$ and $q$ are congruent.
\end{definition}
Now we can define rigidity and flexibility.
\begin{definition}\label{def: rigidity/flexibility}
A framework $(G,p)$ 
is (locally) \defn{rigid} if there is a neighborhood
$U \subset \RR^{dn}$ of $p$ with the property that
if $q\in U$ and $(G,p)$ and $(G,q)$ are equivalent, 
then they are congruent.

A \defn{finite motion} 
of a framework $(G,p)$ is a
one-parameter family of frameworks $(G,p_t)$ 
with $p_0 = p$ and $(G,p_t)$ equivalent to 
$(G,p)$ for all $t\in [0,1)$.  A finite motion is 
\defn{non-trivial} if not all the $(G,p_t)$ are congruent to
$(G,p)$.  A framework is \defn{flexible} if it has
a non-trivial finite motion.
\end{definition}
The definitions of rigid and flexible are, a priori, 
not directly related.  It is straightforward that a 
rigid framework is not flexible.  Asimow and Roth \cite{AR78} proved
the strong converse.
\begin{theorem}\label{thm: not-rigid-flex}
Let $(G,p)$ be a $d$-dimensional framework.  Then if $(G,p)$ is not 
rigid, it is flexible.  Hence, for any $d$ and $G$, every $d$-dimensional 
framework $(G,p)$ is either rigid or flexible.
\end{theorem}

\paragraph{Infinitesimal rigidity}
Theorem \ref{thm: not-rigid-flex} holds unconditionally, 
and its proof relies on a difficult algebraic-geometric result 
(Milnor's curve selection lemma \cite{milnor}).  In the 
generic case, it suffices to study the linearization of rigidity.

\begin{definition}\label{def: infmot}
An \defn{infinitesimal motion} $p'\in \RR^{dn}$ of a $d$-dimensional framework $(G,p)$ is a 
velocity field supported on $p$ such that 
\begin{eqnarray}\label{eq:infmot}
	[p(j) - p(i)]\cdot [p'(j) - p'(i)] = 0 & 
    \text{for all $\{i,j\}\in E$}
\end{eqnarray}

An infinitesimal motion is called \defn{trivial} if it 
arises as the derivative of a rigid motion of $\EE^d$, 
restricted to $p$.  The dimension of the space of trivial infinitesimal 
motions of a framework in  $\EE^d$ with at least $d$ vertices 
is $\binom{d+1}{2}$.

A framework $(G,p)$ in dimension $d$ is \defn{infinitesimally rigid}
if every infinitesimal motion of it is trivial.  Otherwise $(G,p)$
is \defn{infinitesimally flexible}.
\end{definition}
Theorem \ref{thm: ar} from the introduction follows from the following more 
specific statement.
\begin{theorem}[name={\cite{AR78}}]\label{thm:rigidity-is-generic}
Fix a dimension $d$ and let $G$ be a graph with $n\ge d$ vertices.
If a $d$-dimensional framework $(G,p)$ is infinitesimally rigid, then it is rigid. 
If $(G,p)$ is generic and infinitesimally flexible, then it is 
flexible.
\end{theorem}
To study infinitesimal rigidity, we use the rigidity matrix.
\begin{definition}\label{def: rigidity matrix}
Fix a dimension $d$ and let $G$ be a graph with $n\ge d$
vertices.  The \defn{rigidity matrix} $R(p)$ of a $d$-dimensional framework $(G,p)$
is the $m\times dn$ matrix of the system 
\eqref{eq:infmot}, where $p'$ is unknown.  
The space of infinitesimal  motions $M(p)$ 
is the kernel of $R(p)$, and $(G,p)$ is infinitesimally
rigid if and only if $\rank R(p) = dn - \binom{d+1}{2}$.
\end{definition}
We are now ready to set up some important concepts relating to
infinitesimal rigidity.
\begin{definition}\label{def: inf-ind-stress}
Let $(G,p)$ be a framework in dimension $d$.  A vector $\omega$
in the left kernel of the rigidity matrix $R(p)$ is called an 
\defn{equilibrium stress} of $(G,p)$.

If $(G,p)$ has no non-zero equilibrium stress, then $(G,p)$ is 
\defn{independent}.  Otherwise $(G,p)$ is \defn{dependent}.
A framework is independent if and only if its rigidity matrix 
has linearly independent rows.

An edge $e$ of a graph $G$ is called a 
\defn{redundant edge of $(G,p)$} if there is an equilibrium stress 
$\omega$ of the framework $(G,p)$ with $\omega(e)\neq 0$.  (The edge $e$ is called redundant because removing 
$e$ from $G$ does not change the rank of $R(p)$.)

A framework that is both infinitesimally 
rigid and independent is called \defn{isostatic}.
\end{definition}

\subsection{The rigidity matroid}\label{sec: rigmat}
Throughout this paper we will use the standard matroid terminology 
of \defn{rank}, \defn{bases} and \defn{independent sets} 
(see, e.g., \cite{O11} for an introduction to matroids).  

A standard result that is implicit in \cite{AR78} is
\begin{lemma}\label{lem: rigmatroid}
Fix a dimension $d$ and let $G$ be a graph with $n\ge d$ vertices.  If there is any configuration $p$ of $n$ points so that $(G,p)$ is isostatic,
then for every generic $q$, $(G,q)$ is isostatic.
\end{lemma}
As a corollary we obtain a statement first made explicit in \cite{LY82}.
\begin{lemma}\label{lem: rigmat defd}
The rigidity matroid $M_{d,n}$ from Definition 
\ref{def: rigidity matroid} exists and is isomorphic to the 
linear matroid on the rows of the rigidity matrix $R(p)$ of 
the complete graph $K_n$ for any generic choice of $p$.
\end{lemma}
\begin{definition}\label{def: redundant}
Fix a dimension $d$ and let $G=(V,E)$ be a graph with $n$ vertices. The \defn{rank} of $G$  is the rank of $E$ in the rigidity matroid $M_{d,n}$.
An edge $e$ of $G$ is \defn{redundant} if the rank of $G$
and the rank of the graph $G\setminus \{e\}$ obtained from $G$ by removing the edge $e$
are the same in $M_{d,n}$.

A subset $E' = \{e_1, \ldots, e_k\}$ of edges in $G$ is redundant if
$G\setminus E'$ has the same rank as $G$ in $M_{d,n}$.
\end{definition}

We will need the following lemmas.

\begin{lemma}\label{lem: stress redundant 1}
Fix a dimension $d$ and let $G$ be a graph with $n\ge d$ vertices.
An edge $e$ of $G$ is redundant if and only if for any generic 
$d$-dimensional framework $(G,p)$ there is an equilibrium stress $\omega$ with 
coordinate $\omega(e)\neq 0$.
\end{lemma}
\begin{proof}
Suppose there exists a generic $d$-dimensional framework $(G,p)$ with the property that every equilibrium stress $\omega$ of $(G,p)$
satisfies $\omega(e)= 0$. Then the row corresponding to $e$ in $R(p)$ is outside
the span of the other rows.  Hence the rank of $R(p)$ will drop if we remove 
that row.  By Lemma \ref{lem: rigmat defd}, $e$ is not redundant.

Conversely, if for any generic $p$ there is an equilibrium stress $\omega$ of $(G,p)$ with $\omega(e)\neq 0$, then the row of $R(p)$ corresponding to $e$ is a linear combination 
of some of the other rows.  Hence, it is possible to pick a set of edges 
$B$ of $G$ not containing $e$ so that the set of rows corresponding to $B$ are a 
basis for the row space of $R(p)$.  It follows that $e$ is redundant.
\end{proof}
\begin{lemma}\label{lem: stress redundant 2}
Fix a dimension $d$ and let $G$ be a graph.  Then a 
subset of edges $E' = \{e_1, \ldots, e_k\}$ is redundant
if and only if for any generic $d$-dimensional framework $(G,p)$
there are equilibrium stresses $\omega_1, \ldots, \omega_k$ 
so that $\omega_i(e_i)\neq 0$ for $i\in [k]$ and $\omega_{i}(e_j) = 0$
for $i\neq j\in [k]$.
\end{lemma}
\begin{proof}
Define $G_i$ to be the graph obtained by removing all the edges of $E'$ 
except for $e_i$ from $G$.  Note that the set $E'$ is redundant in $G$ if and only 
if $e_i$ is redundant in $G_i$ for each $i$.  Since an equilibrium 
stress $\omega_i$ of $(G_i,p)$ is also an equilibrium stress of $(G,p)$,
with $\omega_i(e_j) = 0$ for $i\neq j\in [k]$, applying Lemma \ref{lem: stress redundant 1}
to each $(G_i,p)$ completes the proof.
\end{proof}

\section{Coordinated frameworks and rigidity}\label{sec:coordinated-frameworks}
The main objects of study in this paper are frameworks 
in which the edges are partitioned into 
coordination classes.  We augment the allowed finite 
motions so that the edge lengths within each class 
may change, but the pairwise differences are preserved.
Thus, the allowed motions are ``coordinated'' within
each coordination class.  

In this section we define coordinated frameworks.  To do this, we 
need to describe the combinatorial and geometric data describing a coordinated
framework and the allowed finite motions.  We then define 
infinitesimal motions of coordinated frameworks and 
derive an associated rigidity matrix.
The development runs in parallel to Section \ref{sec:background}.

\subsection{Combinatorial data}
Fix a parameter $k\in \NN$, which we call the number of 
\defn{coordination classes}, and let $G=(V,E)$ be a 
graph.  A \defn{coordination map} is a function $c : E\to \{0,1,\ldots, k\}$.
The underlying combinatorial structure of a coordinated framework is a 
pair $(G,c)$.  For convenience, we define the notation 
$E_i := c^{-1}(i)$ for $i\in \{0,1,\ldots, k\}$, where $E_0$ is the set of
\defn{uncoordinated edges}, and $E_i$ for $i\in [k]$ is the $i$-th \defn{coordination class}.
Throughout this paper, we assume that $E_i \neq \emptyset$ for all $i=1,\ldots, k$.

We call a pair $(G,c)$, where $G$ is a 
graph and $c: E\to \{0,1,\ldots, k\}$ is a coordination map, a \defn{$k$-coordinated graph}.

\subsection{Coordinated frameworks and rigidity}
Let $(G,c)$ be a $k$-coordinated graph with $n$ vertices.
A \defn{placement} $(p,r)$ of $(G,c)$ is given by a point 
configuration $p$ of $n$ points in dimension $d$ 
and a vector $r\in \RR^k$.  Two
placements $(p,r)$ and $(q,s)$ are \defn{congruent}
if $p$ and $q$ are congruent.

A \defn{coordinated framework} $(G,c,p,r)$ is given 
by a $k$-coordinated graph $(G,c)$ and a placement $(p,r)$.
Two frameworks
$(G,c,p,r)$ and $(G,c,q,s)$ are \defn{equivalent} if
\begin{eqnarray}
	\label{eq:e0-constraint}
	||p(j) - p(i)|| = ||q(j) - q(i)|| & 
    	\text{for all $\{i,j\}\in E_0$} \\
        \label{eq:ej-constraint}
    ||p(j) - p(i)|| + r(\ell) = ||q(j) - q(i)|| + s(\ell) & 
    \text{for all $\{i,j\}\in E_\ell$, with $\ell\in [k]$}
\end{eqnarray}
and they are \defn{congruent} if they are equivalent and 
the placements are congruent.  
Figure~\ref{fig:equiv} shows two equivalent, but not congruent, realizations of
a 2-coordinated graph $(K_4,c)$.

A coordinated framework $(G,c,p,r)$ is 
\defn{rigid} if there is a neighborhood 
$U\subset \RR^{dn}\times \RR^k\cong \RR^{dn + k}$
of $(p,r)$ with the property that if $(q,s)\in U$ and $(G,c,q,s)$ is equivalent
to $(G,c,p,r)$, then the two frameworks are congruent. A coordinated framework is \defn{generic} if $p$ is generic.

A \defn{finite motion} of a coordinated framework $(G,c,p,r)$
is a one-parameter family $(G,c,p_t,r_t)$ with $(p_0,r_0) = (p,r)$
and all the $(G,c,p_t,r_t)$ are equivalent to $(G,c,p,r)$, for $t\in [0,1)$.  A finite motion 
is non-trivial if not all the 
$(G,c,p_t,r_t)$ are congruent to $(G,c,p,r)$.  A coordinated 
framework is \defn{flexible} if it has a non-trivial 
finite motion.

\begin{remark}\label{rem:finite}
Geometrically, what is maintained over a finite 
motion is the  differences in length between pairs
of edges $\{i,j\}$ and $\{u,v\}$
in the same coordination class $E_\ell$, since 
\[
	||p_t(j) - p_t(i)|| + r_t(\ell) - 
    ||p_t(v) - p_t(u)|| - r_t(\ell) =  
    ||p_t(j) - p_t(i)|| - ||p_t(v) - p_t(u)||
\]
does not depend on $r_t(\ell)$, so it must be constant over the motion.
\end{remark}

\begin{figure}[htp]
\begin{center}
\begin{tikzpicture}[very thick,scale=1.4]
\tikzstyle{every node}=[circle, draw=black, fill=white, inner sep=0pt, minimum width=4pt];
\node (p1) at (0,0) {};
\node (p2) at (1,0) {};
\node (p3) at (1,1) {};
\node (p4) at (0,1) {};
\draw[dashed] (p4) -- (p2) (p3) -- (p4) ;
\draw  (p2) -- (p1) (p1) -- (p3);
\draw (p2) -- (p3) (p4) -- (p1) ;
\node [rectangle,draw=white, fill=white] (b) at (-0.25,-0.25) {p(1)};
\node [rectangle,draw=white, fill=white] (b) at (1.25,-0.25) {p(2)};
\node [rectangle,draw=white, fill=white] (b) at (1.25,1.25) {p(3)};
\node [rectangle,draw=white, fill=white] (b) at (-0.25,1.25) {p(4)};
\end{tikzpicture}
\hspace{1.5cm}
\begin{tikzpicture}[very thick,scale=1.4]
\tikzstyle{every node}=[circle, draw=black, fill=white, inner sep=0pt, minimum width=4pt];
\node (p1) at (0,0) {};
\node (p2) at (1,0) {};
\node (p3) at (1,1) {};
\node (p4) at (0.64767,0.761921) {};
\draw[dashed] (p4) -- (p2) (p3) -- (p4) ;
\draw  (p2) -- (p1) (p1) -- (p3);
\draw (p2) -- (p3) (p4) -- (p1) ;
\node [rectangle,draw=white, fill=white] (b) at (-0.25,-0.25) {q(1)};
\node [rectangle,draw=white, fill=white] (b) at (1.25,-0.25) {q(2)};
\node [rectangle,draw=white, fill=white] (b) at (1.25,1.25) {q(3)};
\node [rectangle,draw=white, fill=white] (b) at (0.25,0.8) {q(4)};
\end{tikzpicture}
\end{center}
\vspace{-0.3cm}
\caption{Two equivalent but non-congruent coordinated frameworks $(K_4,c,p,r)$ and $(K_4,c,q,s)$  in the plane with $k=1$, where edges in $E_1$ are 
denoted by dashed lines. The coordinates of the points are  $p(1)=q(1)=(0,0)$, $p(2)=q(2)=(1,0)$, 
$p(3)=q(3)=(1,1)$, $p(4)=(0,1)$ and $q(4)$ has coordinates 
close to $(0.64767,0.761921)$. Also, $r=0$ and $s$ is an 
algebraic number close to $0.574773$
}
\label{fig:equiv}
\end{figure}

\subsection{Coordinated infinitesimal rigidity} 
Assuming that there are no zero-length 
edges in a coordinated framework $(G,c,p,r)$, 
the Jacobian matrix of the system
\eqref{eq:e0-constraint}--\eqref{eq:ej-constraint}
is the linear system
\begin{eqnarray}
	\label{eq:inf0-constraint0}
	\frac{[p(j) - p(i)]\cdot [p'(j) - p'(i)]}{\|p(j) - p(i)\|} = 0 & 
    	\text{for all $\{i,j\}\in E_0$} \\
        \label{eq:infj-constraint0}
    \frac{[p(j) - p(i)]\cdot [p'(j) - p'(i)]}{\|p(j) - p(i)\|} + 
    r'(\ell) = 0 & 
    \text{for all $\{i,j\}\in E_\ell$, with 
    	$\ell\in [k]$}
\end{eqnarray}
so a $(p',r')$ satisfying
\eqref{eq:inf0-constraint0}--\eqref{eq:infj-constraint0}
preserves the coordinated framework's constraints to first order.
Because it will be easier to work with combinatorially, we instead 
define an \defn{infinitesimal motion} of a coordinated framework
$(G,c,p,r)$ to be a pair $(p',r')$ consisting of a
velocity field $p'$ supported on $p$ and a vector
$r'\in \RR^k$ such that
\begin{eqnarray}
	\label{eq:inf0-constraint}
	[p(j) - p(i)]\cdot [p'(j) - p'(i)] = 0 & 
    	\text{for all $\{i,j\}\in E_0$} \\
        \label{eq:infj-constraint}
    [p(j) - p(i)]\cdot [p'(j) - p'(i)] + r'(\ell) = 0 & 
    \text{for all $\{i,j\}\in E_\ell$, with 
    	$\ell\in [k]$}
\end{eqnarray}
We justify this as follows:
\begin{lemma}\label{lem: inf motions make sense}
Let $(G,c,p,r)$ be a coordinated framework such that 
the endpoints of every edge are distinct.  Then 
there is a non-zero solution to \eqref{eq:inf0-constraint0}--\eqref{eq:infj-constraint0}
if and only if there is a non-zero solution to 
\eqref{eq:inf0-constraint}--\eqref{eq:infj-constraint}.
\end{lemma}
\begin{proof}
Suppose that $(p',r')$ satisfies \eqref{eq:inf0-constraint0}--\eqref{eq:infj-constraint0}.
Clearing the denominator in \eqref{eq:infj-constraint0}, we 
get
\[
    [p(j) - p(i)]\cdot [p'(j) - p'(i)] + \|p(j) - p(i)\|r'(\ell) = 0
\]
for all edges $\{i,j\}\in E_\ell$.  If we define the vector
\[
    r'' = \left( \|p(j) - p(i)\|r'(\ell) \right)_{\ell=1}^k
\]
it then follows that $(p',r'')$ satisfies 
\eqref{eq:inf0-constraint}--\eqref{eq:infj-constraint}.
The other direction is similar.
\end{proof}

Define $\ones(c)$ to be the $m\times k$ matrix 
that has as its columns the characteristic vectors
of the $E_\ell$.  Then \eqref{eq:inf0-constraint}--%
\eqref{eq:infj-constraint} is equivalent to 
\begin{equation}\label{eq:inf-rigidity-matrix-form}
	R(p)p' + \ones(c)r' = 0
\end{equation}
where $R(p)$ is the rigidity matrix  of 
the \defn{underlying framework} $(G,p)$.  Note 
that $r$ does not appear in \eqref{eq:inf-rigidity-matrix-form},
so infinitesimal rigidity of $(G,c,p,r)$ depends only on $p$.  Thus,
for analysing infinitesimal rigidity, we may assume that $r=0$.

Since $r'$ can be the zero vector, \eqref{eq:inf-rigidity-matrix-form} is homogeneous. Thus, the
infinitesimal motions form a vector space that
contains a $\binom{d+1}{2}$-dimensional subspace
of motions $(p',\vec{0})$, with $p'$ a trivial
infinitesimal motion of $(G,p)$. This is the space of trivial infinitesimal motions of $(G,c,p,r)$. We define 
$(G,p,c,r)$ to be \defn{infinitesimally rigid}
if these are the only infinitesimal motions, and 
\defn{infinitesimally flexible} otherwise.
The construction in Lemma 
\ref{lem: inf motions make sense} maps vectors 
with $r' = 0$ to vectors with $r'' = 0$ and vice versa,
so the definition of infinitesimal rigidity by 
\eqref{eq:inf0-constraint0}--\eqref{eq:infj-constraint0}
is the same as the one here.

Examples of an infinitesimally flexible and an infinitesimally rigid coordinated framework with $k=1$ and $d=2$ are shown in Figure~\ref{fig:motionexamp}. Note that if $G$ is generically flexible (or even isostatic) in dimenson $d$, then $(G,c,p,r)$ can never be infinitesimally rigid for any $d$-dimensional configuration $p$.

\medskip

\begin{figure}[htp]
\begin{center}
\begin{tikzpicture}[very thick,scale=1.3]
\tikzstyle{every node}=[circle, draw=black, fill=white, inner sep=0pt, minimum width=4pt];
\node (p1) at (0,0) {};
\node (p2) at (1.7,0) {};
\node (p3) at (1.7,1) {};
\node (p4) at (0,1) {};
\draw[dashed](p1)--(p2);
\draw(p3)--(p2);
\draw[dashed](p4)--(p3);
\draw(p4)--(p1);
\draw(p3)--(p1);
\node [rectangle,draw=white, fill=white] (b) at (0.85,-0.6) {(a)};
\end{tikzpicture}
\hspace{0.5cm}
\begin{tikzpicture}[very thick,scale=1.3,pile/.style={thick, ->, >=stealth'}]
\tikzstyle{every node}=[circle, draw=black, fill=white, inner sep=0pt, minimum width=4pt];
\node (p1) at (0,0) {};
\node (p2) at (1.7,0) {};
\node (p3) at (1.7,1) {};
\node (p4) at (0,1) {};
\draw[dashed](p1)--(p2);
\draw(p3)--(p2);
\draw[dashed](p4)--(p3);
\draw(p4)--(p1);
\draw(p3)--(p1);
\draw[gray,pile](p4)--(-0.5,1);
\draw[gray,pile](p2)--(2.2,0);
\node [rectangle,draw=white, fill=white] (b) at (0.85,-0.6) {(b)};
\end{tikzpicture}
\hspace{0.5cm}
\begin{tikzpicture}[very thick,scale=1.3]
\tikzstyle{every node}=[circle, draw=black, fill=white, inner sep=0pt, minimum width=4pt];
\node (p1) at (0,0) {};
\node[draw=black!20!white] (p2) at (1.7,0) {};
\node (p3) at (1.7,1) {};
\node[draw=black!20!white] (p4) at (0,1) {};
\node (p5) at (-0.65,0.77) {};
\node (p6) at (2.35,0.23) {};
\draw[dashed,black!20!white](p1)--(p2);
\draw[black!20!white](p3)--(p2);
\draw[dashed,black!20!white](p4)--(p3);
\draw[black!20!white](p4)--(p1);
\draw(p3)--(p1);
\draw[dashed](p5)--(p3);
\draw[dashed](p6)--(p1);
\draw(p5)--(p1);
\draw(p3)--(p6);
\draw[thick,->,black!80!white](-0.1,1) arc (90:120:1cm);
\draw[thick,->,black!80!white,xshift=1.7cm,yshift=-1cm](0.1,1) arc (270:300:1cm);
\node [rectangle,draw=white, fill=white] (c) at (0.85,-0.6) {(c)};
\end{tikzpicture}
\hspace{0.5cm}
\begin{tikzpicture}[very thick,scale=1.3]
\tikzstyle{every node}=[circle, draw=black, fill=white, inner sep=0pt, minimum width=4pt];
\node (p1) at (0,0) {};
\node (p2) at (1.7,0) {};
\node (p3) at (1.7,1) {};
\node (p4) at (0,1) {};
\draw[dashed](p1)--(p2);
\draw(p3)--(p2);
\draw[dashed](p4)--(p3);
\draw(p4)--(p1);
\draw(p3)--(p1);
\draw[dashed](p2)--(p4);
\node [rectangle,draw=white, fill=white] (c) at (0.85,-0.6) {(d)};
\end{tikzpicture}
\end{center}
\vspace{-0.3cm}
\caption{A coordinated framework with $k=1$ in the plane,
where the edges in $E_1$ are shown dashed (a). This framework
has a non-trivial infinitesimal motion shown in (b) which
extends to a non-trivial finite motion, as indicated in (c).
Adding another bar to the framework in (a) yields an
infinitesimally rigid framework, as shown in (d).}
\label{fig:motionexamp}
\end{figure}

We define the space of \defn{infinitesimal motions}
$M^+(p)$ of $(G,c,p,r)$ to be the space of solutions 
to \eqref{eq:inf-rigidity-matrix-form}.  By rearranging, 
we see that $M^+(p)$ is the kernel of the 
$m\times (dn + k)$ matrix $R^+(p) := (R(p) , \ones(c))$,
which we call the \defn{coordinated rigidity matrix}.

\begin{example}
The  framework $(K_4,c,p,r)$ with $k=1$ shown in Figure \ref{fig:equiv} has the following coordinated rigidity matrix:
\begin{equation*}R^+(p)=
    \begin{bmatrix}
    p(1)-p(2)   & p(2)-p(1) & \mathbf{0}         & \mathbf{0}         & 0   \\
    p(1)-p(3)   & \mathbf{0}         & p(3)-p(1) & \mathbf{0}         & 0  \\
    p(1)-p(4)   & \mathbf{0}         &   \mathbf{0}       & p(4)-p(1) & 0  \\
    \mathbf{0}           & p(2)-p(3) & p(3)-p(2) & \mathbf{0}         & 0  \\
   \mathbf{0}           & p(2)-p(4) & \mathbf{0}        & p(4)-p(2) & 1 \\
    \mathbf{0}           & \mathbf{0}         & p(3)-p(4) & p(4)-p(3) & 1  \\
    \end{bmatrix}
\end{equation*}
where each $p(i)$ is considered a $2$-dimensional row vector. The rank of this matrix is $6$ and hence $(K_4,c,p,r)$ is infinitesimally rigid.

Note that for any $p$, the row rank of this matrix is clearly at most 6. For a $2$-dimensional framework on
four vertices with $k=2$ to be infinitesimally rigid, we would require a column rank of 7. 
\end{example}

\begin{theorem}\label{thm:inf-finite}
Let $(G,c,p,r)$ be a $d$-dimensional coordinated framework.  If 
$(G,c,p,r)$ is infinitesimally rigid, then it is 
rigid.  If $(G,c,p,r)$ is generic and infinitesimally
flexible, then it is flexible.
\end{theorem}
Theorem \ref{thm:inf-finite} can be established using 
differential-geometric 
arguments along the lines of \cite{AR78}.  We omit the 
(standard) details, which can be found in
\cite{HS19}.
The combinatorial perspective on this is:
\begin{corollary}\label{cor:generic}
For every $k$-coordinated graph $(G,c)$ and dimension $d$, either every $d$-dimensional generic
coordinated framework $(G,c,p,r)$ is rigid or every 
 $d$-dimensional generic coordinated framework $(G,c,p,r)$ is flexible.
\end{corollary}
In light of Corollary \ref{cor:generic}, we define $(G,c)$ to be \defn{generically rigid} in dimension $d$
if there is a generic $d$-dimensional framework $(G,c,p,r)$ that is rigid, and otherwise
\defn{generically flexible}. 

We say that a coordinated framework $(G,c,p,r)$ is \defn{independent} 
if the coordinated rigidity matrix $R^+(p)$ has independent rows. Moreover,
$(G,c,p,r)$ is \defn{isostatic} if it is infinitesimally rigid and independent. Similarly, $(G,c)$ is \defn{generically independent (isostatic)} in dimension $d$ if there is a generic $d$-dimensional framework $(G,c,p,r)$ that is independent (isostatic).

\section{Generic coordinated rigidity}\label{sec:bridges}
In this section we develop the generic theory for coordinated  rigidity and show that in all dimensions $d$, generic coordinated rigidity can be characterized in terms of redundant rigidity in the standard $d$-dimensional rigidity matroid.

\subsection{Main theorem}
Given a family  of sets, $\cE$, a \defn{transversal} of $\cE$ is a set containing exactly one element from each member of $\cE$.

The following is our main theorem.
\begin{theorem}\label{thm:all-k}
For $d\ge 1$ and $k\ge 1$, $(G,c)$ is generically 
rigid in dimension $d$ if and only if $G$ is generically rigid in dimension $d$ and 
some transversal $\{e_1, \ldots, e_k\}$ of the 
coordination classes $E_1, \ldots, E_k$ is redundant in $M_d(G)$.
\end{theorem}

Theorem \ref{thm:inf-finite} implies that generic coordinated rigidity is a matroidal property, and we may use Theorem \ref{thm:all-k}  to describe the coordinated rigidity matroid.
Before we give the proof of Theorem \ref{thm:all-k}, we reformulate this result in terms of matroid unions. We need the following definitions.

Let $E$ be a finite set and let $\cE = \{E_1, \ldots, E_k\}$ 
be a collection of disjoint subsets of $E$.  The 
\defn{transversal matroid} $T_E(\cE)$ on $E$ induced by the $E_i$
has as its bases the sets
\[
    \{ \{e_1, \ldots, e_k\} : \text{$e_i \in E_i$ for all $i\in [k]$}\}
\] See, e.g., \cite{bry,O11}.

If $M_1$ and $M_2$ are two matroids on a common ground set $E$, then the 
\defn{matroid union} $M_1\vee M_2$ is defined as the matroid on $E$ with the property that a subset $F$ is independent in $M_1\vee M_2$ if and only if it has the form $F=F_1\cup F_2$, where $F_i$ is independent in $M_i$ for $i=1,2$.

With these definitions, we have the following matroidal formulation of
Theorem \ref{thm:all-k}.
\begin{theorem}\label{thm:all-k matroid}
Let $(G,c)$ be a $k$-coordinated graph, and let 
$T_E(\cE)$ be the transversal matroid on $E$ induced by the coordination classes $\cE=\{E_1,\ldots, E_k\}$.  Then the 
$d$-dimensional $k$-coordinated rigidity matroid of $(G,c)$ is the union 
$M_d(G)\vee T_E(\cE)$ of the $d$-dimensional rigidity matroid of $G$
and the matroid $T_E(\cE)$.
\end{theorem}
\begin{proof}
We show that the two matroids have the same bases.
The generically isostatic $k$-coordinated spanning subgraphs of $(G,c)$ form the bases of the $d$-dimensional coordinated 
rigidity matroid of $(G,c)$.  Let $(H,c')$ be such a graph, with $H = (V,F)$ and $c'=c|_F$.  By Theorem \ref{thm:all-k}, $H$ must have exactly 
$dn - \binom{d+1}{2} + k$ edges, including 
a transversal $F'=\{e_1, \ldots, e_k\}$ of $\cE$ that is redundant in 
$M_d(H)$.  Thus $F\setminus F'$ is independent in $M_d(G)$,
and hence a basis of $M_d(G)$, and $\{e_1, \ldots, e_k\}$ is a basis of $T_E(\cE)$.
Conversely, Theorem \ref{thm:all-k} implies that any $k$-coordinated graph arising in this way is a basis of the $d$-dimensional $k$-coordinated rigidity matroid of $(G,c)$.

Hence, the $d$-dimensional coordinated rigidity matroid has
the claimed matroid union structure.
\end{proof}

To prove Theorem \ref{thm:all-k} we need the following specialized fact from matroid theory. This can be found in \cite[Proposition 7.6.14]{bry}, for example.

\begin{lemma} \label{lem:matunion}
Let $M_1$ and $M_2$ be two linearly representable  matroids (over $\RR$)  on the same ground set. Then the matroid union $M_1\vee M_2$ is also linearly
representable, and a representation may be obtained
by a matrix of the form $(A,DB)$ where the rows of 
$A$ represent $M_1$, the rows of $B$ represent $M_2$, and $D$ is a diagonal matrix of algebraically independent transcendentals.  
\end{lemma}

We also need  some specialized lemmas regarding equilibrium stresses.
Let $(G,p)$ be a framework with at least $d$ vertices.  We say that 
the \defn{tied down} rigidity matrix $R^{i_1, \ldots, i_d}(p)$ is the matrix obtained from $R(p)$ by throwing away $d - j + 1$ columns associated with $p(i_j)$.
White and Whiteley \cite{WW83} showed that $\omega$ is an equilibrium stress of $(G,p)$ if and only if $\omega R^{i_1, \ldots, i_d}(p) = 0$ for any choice 
of the $i_j$.

\begin{lemma}[\cite{WW83}]\label{lem: stress coeff special}
Suppose that $(G,p)$ is a generic framework so that $(G\setminus \{e\}, p)$ is isostatic (and hence $e$ is  redundant).  Then 
$(G,p)$ has a unique equilibrium stress $\omega$  where $\omega(e) = 1$ and for all other edges $f$, we have
\[
    \omega(f) = (\det R_{e\to \times}^{i_1, \ldots, i_d}(p))^{-1}\det(R^{i_1, \ldots, i_d}_{f\to e}(p))
\]
where the $i_j$ are any tie-down vertices and $R^{i_1, \ldots, i_d}_{f\to e}$ is obtained by removing the row corresponding to $e$ and then 
replacing the row corresponding to $f$ with it and $R^{i_1, \ldots, i_d}_{e\to \times}$ by simply dropping the row corresponding to $e$.
\end{lemma}
\begin{proof}
White and Whiteley \cite{WW83} show that $\omega$ is independent of this kind of standard tie down, and then 
observe that, in this case, Cramer's rule gives the claimed form, since the row corresponding to $e$ is 
in the span of the others, which are a row basis.
\end{proof}

Next we show that if the length of the edge $e$ gets very small, then the stress coefficient corresponding to $e$ becomes very large compared to all other coefficients.
\begin{lemma}\label{lem: crush}
Suppose that $(G,p)$ is a generic framework so that $(G\setminus \{e\}, p)$ is isostatic (and hence $e$ is redundant).  Let $e = \{i_1,i_2\}$ and 
define $p^t$ to be like $p$ 
except $p^t(i_2) = tp(i_2) + (1 - t)p(i_1)$.  
Let $\omega^t$ be the equilibrium stress of $(G,p^t)$ with $\omega^t(e) = 1$.  Then for all other edges $f$, we have for generic $p_t$,
\[
    |\omega^t(f)| \to 0
\]
as $t\to 0$.
\end{lemma}
\begin{proof}
Since $(G\setminus \{e\}, p)$ is isostatic, $(G,p)$ has a unique equilibrium stress, up to a scalar. Moreover, the stress coefficient of $e$ is non-zero.
Note that for any polynomially defined property, such as the rank of the rigidity matrix or the stress coefficients being non-zero, all but finitely many of the frameworks $(G,p^t)$ have this property.
So there exists a $t_0$ so that for all $t<t_0$, all frameworks $(G,p^t)$ are isostatic and hence have a unique equilibrium stress (up to a scalar) that is non-zero on $e$. In the following we assume that $t<t_0$.

We fix the stress coefficient $\omega^t(e)$ to be $1$ for all $t$, so that each $(G,p^t)$ has a unique equilibrium stress. 
Consider the stress coefficients of all edges $f$ other than $e$. 
A formula for $\omega^t(f)$ is given in Lemma~\ref{lem: stress coeff special}.
Since the lengths and directions of the edges of $G \setminus \{e\}$ only change a bounded amount as $t\to 0$, and the length of the edge $e$ approaches zero as $t\to 0$, the determinant in the numerator in the formula in Lemma~\ref{lem: stress coeff special} approaches zero as $t\to 0$. 

Let us now consider the determinant  in the denominator in the formula in Lemma~\ref{lem: stress coeff special}.
If the limit framework $(G\setminus\{e\},p^0)$ is isostatic, then the denominator is bounded away from zero for all $t$ and the result follows. If $(G\setminus\{e\},p^0)$ is not isostatic, then both the numerator and the denominator may converge to zero, but we will show that in this case the numerator approaches zero at a faster rate than the denominator, which then again gives the result.

Without loss of generality, we may assume that $p^t(i_1)$ lies at the origin and $p^t(i_2)=(t,0,\ldots, 0)$, since the frameworks $(G,p^t)$ are then still `quasi-generic', that is, congruent to generic frameworks. We may further assume that all the coordinates of $p^t(i_1)$ and all except the first coordinate of $p^t(i_2)$ are tied down (so the associated columns in the tied down rigidity matrices in Lemma~\ref{lem: stress coeff special} are missing) and the first coordinates of all vertices adjacent to $i_2$ are not tied down.  

Suppose first that 
$f=\{i_2,i_3\}$ and let $x_3$ be the first coordinate of $p^t(i_3)$ (which is independent of $t$). By genericity, $x_3\neq 0$.
Let us first consider the matrix $R^{i_1, \ldots, i_d}_{f\to e}(p^t)$ in the numerator.
We may assume that the first row of this matrix corresponds to $e$ and hence is the vector $(t,0,\ldots, 0)$. Let the matrix obtained from $R^{i_1, \ldots, i_d}_{f\to e}(p^t)$ by removing the first row be denoted by $A$, and let $C(t)$ be the cofactor of the $(1,1)$-entry of $R^{i_1, \ldots, i_d}_{f\to e}(p^t)$. Then $\det(R^{i_1, \ldots, i_d}_{f\to e}(p^t))=tC(t)$. Let us now consider the matrix  $R^{i_1, \ldots, i_d}_{e\to \times}(p^t)$ in the denominator. Since a reordering of the rows and columns does not change the degrees of the monomials in the determinants, we may assume that the first row corresponds to the edge $f$ (and hence the $(1,1)$-entry is $(t-x_3)$) and that the matrix obtained from $R^{i_1, \ldots, i_d}_{e\to \times}(p^t)$ by removing the first row is again $A$. Then  $\det(R^{i_1, \ldots, i_d}_{e\to \times}(p^t))=(t-x_3)C(t)+D(t)$ for some polynomial $D(t)$.
Since this polynomial contains the expression $x_3C(t)$, where $x_3$ is a non-zero constant, it follows that the lowest degree of the non-zero terms in the denominator is strictly smaller than the degree of any of the non-zero terms in the numerator. Thus, $|\omega^t(f)|\to 0$ for every edge $f\neq e$ that is incident with $i_2$.

Suppose next that there is an edge $f$ of $G$ that is not incident with $i_2$ and has the property that $|\omega^t(f)|$ is bounded away from zero as $t\to 0$. Then, by continuity, the limit framework $(G\setminus\{e\},p^0)$ has an equilibrium stress $\omega^0$  that is non-zero on $f$ and -- by the argument above -- zero on all edges that are incident with $i_2$. But this says that the framework obtained from $(G,p^0)$ by removing the vertex $i_2$ and its incident edges has a non-zero equilibrium stress. This is a contradiction because the vertices of this framework are in quasi-generic position (since they remained fixed as $t\to 0$) and its underlying graph $G\setminus\{i_2\}$ is a subgraph of the isostatic graph $G\setminus \{e\}$. 

So, as claimed, the magnitude of all stress coefficients of $\omega^t$ other than the one for $e$ approach zero as $t\to 0$. \end{proof}

We are now ready to prove Theorem \ref{thm:all-k}.

\begin{proof}[Proof of Theorem \ref{thm:all-k}]
Let $(G,c,p,r)$ be generic. Recall that for analysing infinitesimal rigidity, we may assume that $r=0$. We first prove necessity, so we suppose that $(G,c,p,0)$ is infinitesimally rigid. By Lemma \ref{lem:matunion}, the matrix $(R(p),D\ones(c))$ is a linear representation for the matroid union $M_d(G)\vee T_E(\cE)$ of the $d$-dimensional
generic rigidity matroid of $G$ and the transversal matroid on $E$ induced by the coordination classes $\cE=\{E_1,\ldots, E_k\}$, where 
$D$ is an $m\times m$ diagonal matrix of algebraically independent transcendentals. Since the coordinated rigidity matrix $R^+(p)$ of $(G,c,p,0)$
has the form $(R(p),\ones(c))$ it follows that any independent set in the $d$-dimensional $k$-coordinated rigidity matroid of $(G,c)$ must also be independent in $M_d(G)\vee T_E(\cE)$. So the rank of 
$R^+(p)$ is upper-bounded by 
\begin{equation}\label{eq:max}
    \max_{E'\subset E}
    \left\{
        \rank_{M_d(G)}\left(E\setminus E'\right)
        + 
        \rank_{T_E(\cE)}(E')
    \right\} \le
    dn - \binom{d+1}{2} + k.
\end{equation}
Since $(G,c,p,0)$ is infinitesimally rigid, we have 
equality throughout. Let $E'$ be a subset of $E$ that gives the maximum $dn - \binom{d+1}{2} + k$ in \eqref{eq:max}, and let $G'$ be the graph induced by $E\setminus E'$.  Since 
\[
    \rank_{M_d(G)}\left(E\setminus E'\right) = dn - \binom{d+1}{2}
\]
the framework $(G',p)$ is infinitesimally rigid.  This makes 
the edges in $E'$ redundant.  Since  
$\rank_{T_E(\cE)}(E')=k$, $E'$ contains a transversal of $\cE$.

For sufficiency, we suppose that there is a spanning subgraph $H$ of $G$ so that $H$ is generically isostatic and $F = E\setminus E(H)$
contains a transversal of the coordination classes $\cE$.  Call the transversal $T = \{e_1,\ldots, e_k\}$, and let $G'$ be the graph obtained from $H$ by adding the edges of $T$ to $H$.  We show that $(G',c)$ is generically isostatic by constructing a coordinated framework $(G',c,p, 0)$ that is isostatic.  The intuitive idea is that we want to find a 
generic position $p$ so that, if we consider the equilibrium stresses $\omega_1, \ldots, \omega_k$ of $(G',p)$ from the statement of Lemma~\ref{lem: stress redundant 2}, the entries not in the left-hand $k\times k$ block of the $k\times m$ matrix $W$ that has 
the $\omega_i$ as rows are very small.  In this case $X = W\ones(c)$ is 
diagonally dominant,
and so full rank.  The idea for doing this latter step is to make the lengths of the 
edges $e_1, \ldots, e_k$, one after the other, all very short.  By Lemma \ref{lem: crush}, this makes the stress coefficient corresponding to $\omega_i$ on $e_i$ relatively large for each $i$.  The transversal structure is important, since 
the stresses $\omega_i$ don't need the other transversal edges at all.  
This leaves only the problem that, as we make  a transversal edge $e_i$ short, the stress coefficients corresponding to the equilibrium stresses $\omega_1,\ldots, \omega_{i-1}$ for some  non-transversal (but coordinated) edges may increase again. We get around that by using Lemma \ref{lem: crush} iteratively, so that $e_1$ is much shorter than $e_2$ and so on.  What follows gives the details.

We start with a generic framework $(H,p)$. Set $H_1 = H \cup \{e_1\}$.  By genericity, $e_1$ is 
redundant in $(H_1,p)$. Let $e_1=\{i_1,i_2\}$. By Lemma \ref{lem: crush}, for every $\eps_1>0$, we may shrink the length $|e_1|$ of the edge $e_1$ by moving the vertex $i_2$ towards $i_1$ (while fixing the positions of all other vertices),  until we obtain a generic configuration $p^t$ with $|e_1|=t$ so that the unique equilibrium stress $\omega_1^t$ of $(H_1,p^t)$ has $\omega_1^t(e_1) = 1$ and coefficients on all other edges of magnitude at most $\eps_1$. For each $i=1,\ldots, k$, let $\kappa_i$ be the size of the $i$th coordination class of the graph $(H,c|_{E(H)})$. Then, since we may assume that $\eps_1$ is sufficiently small, the first entry of $\omega_1^t \ones(c)$ has a magnitude of at least $1-\kappa_1\eps_1$, and the other entries of $\omega_1^t \ones(c)$ have a respective magnitude of at most $\kappa_i\eps_1$ for $i=2,\ldots, k$. 
So as $t$ approaches zero, the magnitude of the first entry of $\omega_1^t \ones(c)$ stays bounded away from $0$, whereas the magnitudes of the other entries approach $0$.

Now we repeat the process on the graph $H_2 = H\cup \{e_2\}$
starting from $(H,p^t)$. Let $i_3\neq i_1,i_2$ be a vertex incident with $e_2$.
Then, again by Lemma \ref{lem: crush}, for every $\eps_2>0$, we may shrink the length $|e_2|$ of the edge $e_2$ by moving the vertex $i_3$ (while fixing the positions of all other vertices) until we obtain a generic configuration $p^{t,s}$ with $|e_2|=s$ so that the unique equilibrium stress $\omega_2^{t,s}$ of $(H_2,p^{t,s})$ has $\omega_2^{t,s}(e_2) = 1$ and coefficients on all other edges of magnitude at most $\eps_2$. In fact, by continuity of the stress coefficients,  if $|e_1|=t$ was chosen sufficiently small, then there exists a generic configuration $p^{t,s}$ with $|e_2|=s$ so that for all $t'\leq t$, the generic configuration $p^{t',s}$ with $|e_1|=t'$ has the property that the unique equilibrium stress $\omega_2^{t',s}$ of $(H_2,p^{t',s})$ has $\omega_2^{t',s}(e_2) = 1$ and coefficients on all other edges of magnitude at most $\eps_2$. 
(Note that since $e_1$ is not an edge of $H_2$, a small perturbation of $|e_1|=t$ gives rise to a bounded perturbation of the stress coefficients of $\omega_2^{t,s}$.)

Because of the transversal structure, as $s$ approaches zero, the magnitude of the second entry of $\omega_2^{t,s}\ones(c)$ stays bounded away from $0$, whereas the magnitudes of the other entries approach $0$. Let $\eps_2$ be chosen sufficiently small so that the magnitude of the second entry of $\omega_2^{t,s} \ones(c)$ is  larger than the sum of the magnitudes of the other entries.
So the vector $\omega_2^{t,s} \ones(c)$ has the desired structure, and, as mentioned above, we may assume that $t$ is sufficiently small so that the vector $\omega_2^{t',s} \ones(c)$ also has the desired structure for all $t'\leq t$.

As we carry out this second step, the equilibrium stress $\omega_1^t$ changes continuously to $\omega_1^{t,s}$, and the magnitude of the stress coefficients on the edges of $H$ may be increasing again (while the stress coefficient on $e_1$ is held fixed at $1$). So to keep the desired structure for the vector $\omega_1^{t,s} \ones(c)$ we may need to choose a $t'$ that is smaller than $t$.

At the end of the second step, when the configuration $p^{t,s}$ is reached, the magnitude of each of the stress coefficients of  $\omega_1^t$ on the edges of $H$ will have changed by at most some positive factor $\Delta_{t,s}$. This factor changes as $t$ is replaced with $t'\leq t$, but 
by continuity of the stress coefficients, there also exists a constant $\Delta_s>0$  such that for all $t'\leq t$ we have $|\omega_1^{t',s}(f)|\leq \Delta_s |\omega_1^{t'}(f)|$ for all edges $f$ of $H$. Since we may assume that $\eps_1$ is sufficiently small, the first entry of the vector  $\omega_1^{t,s} \ones(c)$ has a magnitude of at least $1-\kappa_1\Delta_s\eps_1$ at the end of the second step, 
 and the other entries have a respective magnitude of at most $\kappa_i\Delta_s\eps_1$, $i=2,\ldots, k$.

So if the vector $\omega_1^{t,s} \ones(c)$ no longer has the desired structure, then we may choose a suitably smaller $\eps_1$  so that for a corresponding  suitably small $t'< t$  the vector $\omega_1^{t',s} \ones(c)$ does have the desired structure, i.e., the magnitude of the first entry of the vector  $\omega_1^{t',s} \ones(c)$  is still  larger than the  sum of the magnitudes of the other entries at the end of the second step. So we can guarantee that both the vector $\omega_1^{t',s} \ones(c)$ and the vector $\omega_2^{t',s} \ones(c)$ have  the desired structure for the configuration $p^{t',s}$ at the end of the second step.

Iterating this process, we see that we may choose constants $\eps_1 \ll \eps_2 \ll \cdots \ll \eps_k$ so that eventually we arrive at a configuration $p^k$ which is generic and has $X$ diagonally 
dominant.

We conclude that, generically, the matrix 
\[
    [R(p) \quad \ones(c)]
\]
has empty co-kernel, and hence rank $dn - \binom{d+1}{2} + k$. This completes the proof of sufficiency.
\end{proof}

\subsection{Examples}
Figure \ref{fig:mainthm} shows an example of a rigid coordinated 
framework with two coordination classes in dimension $2$.  The edges
$e_1$ and $e_2$ form a redundant transversal of the coordination classes $E_1$ and $E_2$ and certify 
generic rigidity by Theorem \ref{thm:all-k}.  The edges $f_1$ and $f_2$ are another transversal of $E_1$ and $E_2$, but they are not redundant. 
\begin{figure}[htp]
\begin{center}
\begin{tikzpicture}[very thick,scale=1]
\node at (2.05,3) {};
\node at (1.85,-1) {};
\node at (0.65,1.05) {$e_1$};
\node at (2.75,1.7) {$e_2$};
\node at (1.75,1.15) {$f_1$};
\node at (2.8,0.9) {$f_2$};
\tikzstyle{every node}=[circle, draw=black, fill=white, inner sep=0pt, minimum width=4pt];
\node (p1) at (0.7,2.3) {};
\node (p2) at (1.9,0) {};
\node (p3) at (1,0.9) {};
\node (p4) at (0,0.8) {};
\node (p5) at (2,2) {};
\node (p6) at (2.5,1.2) {};
\node (p7) at (3.6,1) {};
\draw[dashed](p1)--(p2) -- (p5);
\draw(p3)--(p2) -- (p4);
\draw[dashed](p4)--(p3) ; 
\draw(p4)--(p1) ;
\draw(p3)--(p1);
\draw[dotted] (p1) -- (p5) -- (p7);
\draw[dotted] (p7) -- (p6);
\draw (p5) -- (p6) -- (p2) -- (p7);
\node [rectangle,draw=white, fill=white] (b) at (1.25,-0.6) {(a)};
\end{tikzpicture}
\hspace{0.5cm}
\begin{tikzpicture}[very thick,scale=1,pile/.style={thick, ->, >=stealth'}]
\node at (2.05,3) {};
\node at (1.85,-1) {};
\tikzstyle{every node}=[circle, draw=black, fill=white, inner sep=0pt, minimum width=4pt];
\node (p1) at (0.7,2.3) {};
\node (p2) at (1.9,0) {};
\node (p3) at (1,0.9) {};
\node (p4) at (0,0.8) {};
\node (p5) at (2,1.8) {};
\node (p6) at (2.5,1.2) {};
\node (p7) at (3.6,1) {};
\draw (p1)--(p2) -- (p5);
\draw (p3)--(p2) -- (p4);
\draw (p4)--(p1) ;
\draw (p3)--(p1);
\draw (p1) -- (p5) ;
\draw (p2) -- (p6);
\draw (p5) -- (p6) -- (p7) -- (p2) ;
\node [rectangle,draw=white, fill=white] (b) at (1.25,-0.6) {(b)};
\end{tikzpicture}
\hspace{0.5cm}
\begin{tikzpicture}[very thick,scale=1,pile/.style={thick, ->, >=stealth'}]
\node at (2.05,3) {};
\node at (1.85,-1) {};
\tikzstyle{every node}=[circle, draw=black, fill=white, inner sep=0pt, minimum width=4pt];
\node (p1) at (0.7,2.3) {};
\node (p2) at (1.9,0) {};
\node (p3) at (1,0.9) {};
\node (p4) at (0,0.8) {};
\node (p5) at (2,1.8) {};
\node (p6) at (2.5,1.2) {};
\node (p7) at (3.6,1) {};
\draw (p1)--(p2) ;
\draw (p3)--(p2) -- (p4) --  (p3);
\draw (p4)--(p1) ;
\draw (p3)--(p1);
\draw (p1) -- (p5) -- (p7);
\draw (p2) -- (p6);
\draw (p5) -- (p6)  (p7) -- (p2) ;
\node [rectangle,draw=white, fill=white] (b) at (1.25,-0.6) {(c)};
\end{tikzpicture}
\end{center}
\vspace{-0.4cm}
\caption{A coordinated framework in the plane with $k=2$, where edges in $E_1$ and $E_2$ are 
indicated by dashed and dotted lines, respectively (a).
Removing the redundant edge pair $e_1,e_2$ results in a graph that is
rigid as an uncoordinated framework (b), satisfying the conditions of 
Theorem \ref{thm:all-k} for the coordinated framework in (a) to be rigid.
Removing the pair $f_1, f_2$, results in 
the flexible uncoordinated framework (c). 
}
\label{fig:mainthm}
\end{figure}

Figure~\ref{fig:circuits} shows two examples of 2-coordinated  graphs
that yield flexible frameworks for generic configurations in dimension 2.  In 
both cases, there is no transversal of the coordination classes which is redundant.

\begin{figure}[htp]
\begin{center}
\begin{tikzpicture}[very thick,scale=1]
\node at (0,-2) {};
\tikzstyle{every node}=[circle, draw=black, fill=white, inner sep=0pt, minimum width=4pt];
\node (p1) at (-0.5,0) {};
\node (p2) at (-2,1) {};
\node (p3) at (-2,-1) {};
\node (p4) at (-1.3,0) {};
\node (p5) at (0.5,0) {};
\node (p6) at (2,1) {};
\node (p7) at (2,-1) {};
\node (p8) at (1.3,0) {};
\draw[dotted](p1)--(p5) (p2) -- (p6) (p3) -- (p7);
\draw[dashed](p3)--(p4)  (p7) -- (p8);
\draw (p1) -- (p2) -- (p3) -- (p1) -- (p4) -- (p2) ;
\draw (p5) -- (p6) -- (p7) -- (p5) -- (p8) -- (p6) ;
\node [rectangle,draw=white, fill=white] (b) at (0,-1.75) {(a)};
\end{tikzpicture}
\hspace{1.5cm}
\begin{tikzpicture}[very thick,scale=1]
\node at (0,-2) {};
\tikzstyle{every node}=[circle, draw=black, fill=white, inner sep=0pt, minimum width=4pt];
\node (p1) at (0,1.2) {};
\node (p2) at (-2,0) {};
\node (p3) at (-0.75,-0.2) {};
\node (p4) at (0,-1.25) {};
\node (p5) at (0.75,-0.2) {};
\node (p6) at (2,0) {};
\node (p7) at (0,0.2) {};
\draw[dotted] (p1) -- (p7);
\draw[dashed] (p3) -- (p7);
\draw[dashed] (p7) -- (p5);
\draw (p6) -- (p1) (p6) -- (p4) (p6) -- (p5);
\draw (p5) -- (p1) (p5) -- (p4);
\draw (p4) -- (p2) (p4) -- (p3);
\draw (p3) -- (p1) (p3) -- (p2);
\draw (p2) -- (p1);
\node [rectangle,draw=white, fill=white] (b) at (0,-1.75) {(b)};
\end{tikzpicture}
\end{center}
\vspace{-0.3cm}
\caption{A pair of generically flexible 2-coordinated graphs, where edges in $E_1$ and $E_2$ are 
indicated by dashed and dotted lines, respectively. 
In the graph shown in (a) none of the edges in $E_2$ is redundant in $M_2(G)$.
In the graph shown in (b), each of the edges in $E_1$ and $E_2$ is 
redundant in $M_2(G)$, but no pair of them is.
}
\label{fig:circuits}
\end{figure}

Finally, Figure~\ref{fig:spheregraph} shows an example of a flexible coordinated framework $(G,c,p,r)$ with three coordination classes in dimension 3 whose underlying framework $(G,p)$ is rigid.  There is only one edge in each coordination class (the edges joining $p_3$ with $p_6, p_7$ and $p_8$) and if we remove these three edges, we clearly have a flexible uncoordinated framework, as the line through $p_4$ and $p_5$ acts like a hinge. Thus, the three edges in $E_1\cup E_2\cup E_3$ are not redundant.

We note that $p_3$ is a cone vertex in this framework, as it is joined to all other vertices. Thus, as described in the introduction, we may think of this framework as a spherical framework, with $p_3$ fixed at the centre of the sphere. In this model, the  uncoordinated edges incident to $p_3$ are assumed to have unit length, and the other three edges incident to $p_3$ are allowed to have arbitrary length. This example was given in \cite[Figure 1]{NSTW15} to show that the sparsity counts established there are not always sufficient for the generic rigidity of such $3$-dimensional spherical frameworks. However, as described above, it follows immediately from Theorem~\ref{thm:all-k} that $(G,c)$ is generically flexible, and hence we may deduce that the (non-generic) spherical framework must at least be infinitesimally flexible.

\begin{figure}[htp]
\begin{center}
\begin{tikzpicture}[very thick,scale=1]
\tikzstyle{every node}=[circle, fill=white, inner sep=0pt, minimum width=5pt];
\linespread{1.0}
\node [circle, shade, ball color=black!40!white, inner sep=0pt, minimum width=7pt](p1) at (-2.4,0) {};
\node [circle, shade, ball color=black!40!white, inner sep=0pt, minimum width=7pt](p2) at (-1.2,-0.2) {};
\node [circle, shade, ball color=black!40!white, inner sep=0pt, minimum width=7pt](p3) at (-0.2,0.2) {};
\node [circle, shade, ball color=black!40!white, inner sep=0pt, minimum width=7pt](p4) at (-1,1.7) {};
\node [circle, shade, ball color=black!40!white, inner sep=0pt, minimum width=7pt](p5) at (-0.8,-1.7) {};
\node [circle, shade, ball color=black!40!white, inner sep=0pt, minimum width=7pt](p6) at (1.8,0.2) {};
\node [circle, shade, ball color=black!40!white, inner sep=0pt, minimum width=7pt](p7) at (3,0.6) {};
\node [circle, shade, ball color=black!40!white, inner sep=0pt, minimum width=7pt](p8) at (3,-0.9) {};

\draw (p1) -- (p2)node[rectangle, draw=white, anchor=south west, below left=3pt] {$p_2$};
\draw (p2) -- (p3) node[rectangle, draw=white,anchor=south, below right=3pt] {$p_3$};
\draw [thick](p3) -- (p4) node[rectangle, draw=white,anchor=south, above=5pt] {$p_4$};
\draw [thick](p4) -- (p1) node[rectangle, draw=white,anchor=south, left=5pt] {$p_1$};
\draw(p4)--(p2);
\draw[thick](p3)--(p1);
\draw[thick] (p5)node[rectangle, draw=white,anchor=south, below=5pt] {$p_5$} -- (p1);
\draw (p5) -- (p2);
\draw[thick] (p5) -- (p3);
\draw[dotted] (p6)node[rectangle, draw=white,anchor=south, above=7pt] {$p_6$} -- (p3);
\draw (p6) -- (p4);
\draw (p6) -- (p5);
\draw (p6) -- (p7)node[rectangle, draw=white,anchor=east, right=5pt] {$p_7$};
\draw (p6) -- (p8)node[rectangle, draw=white,anchor=east, right=5pt] {$p_8$};
\draw (p7) -- (p8);
\draw[dashed] (p7) -- (p3);
\draw (p7) -- (p4);
\draw (p7) -- (p5);
\draw[dash dot] (p8) -- (p3);
\draw (p8) -- (p4);
\draw (p8) -- (p5);
\end{tikzpicture}
\end{center}
\vspace{-0.3cm}
\caption{A coordinated framework in $3$-space with $k=3$. The edges in $E_1,E_2$ and $E_3$ are indicated by dashed, dotted and dashed-dotted lines, respectively. The framework is flexible by Theorem~\ref{thm:all-k}, since the three edges in $E_1\cup E_2\cup E_3$ are not redundant.}
\label{fig:spheregraph}
\end{figure}

\section{Closing remarks and further work} \label{sec:closing}
In this paper we set up a rigidity model for a type of coordinated edge motions 
and characterized the coordinated rigid graphs in terms of redundant rigidity in the standard rigidity matroid.
To finish up, we discuss some outstanding issues and open questions.

\subsection{Algorithms}
Note that Theorem \ref{thm:all-k matroid} implies that in dimensions $d=1,2$, there is a deterministic, 
polynomial time algorithm to check whether a $k$-coordinated graph $(G,c)$ is generically
rigid in dimension $d$.  Since we have deterministic independence oracles for the matroids
$M_{d,n}$ when $d=1,2$ \cite{BJ03b, LS08} and $T_E(\cE)$, Edmonds's algorithm \cite{E65} 
yields a deterministic polynomial time algorithm for $M_{d,n}\vee T_E(\cE)$ for these $d$ and any $k$.
Combining Edmonds's algorithm \cite{E65} and the pebble game \cite{LS08} in a naive way 
results in a running time bound that is $\omega(n^3)$.  This means that 
the deterministic 
algorithm is slower than a randomized 
algorithm based on checking the rank of the rigidity matrix over a finite field (see, 
e.g., \cite{GHT}).

For $d\ge 3$, no deterministic polynomial time algorithm for checking independence in the
generic rigidity matroid $M_{d,n}$ is known.  Since we need a randomized algorithm 
as a subroutine, using Edmonds's algorithm does not reduce the complexity.

A more refined complexity analysis in dimensions $d=1,2$ or algorithms for finding 
rigid components in coordinated frameworks would be interesting.

\subsection{Recursive constructions}

Recursive graph construction moves (such as the Henneberg construction moves) which preserve  generic rigidity play an important role in combinatorial rigidity theory (see, e.g., \cite{NR14}). These moves can
be used to prove key theorems such as the theorem of Laman and Pollaczek-Geiringer, to analyze graphs for generic rigidity, and to generate classes of generically rigid or isostatic graphs in all dimensions.

In the case when $d=2$, a ``Henneberg-type'' characterization of
generically isostatic coordinated graphs with one coordination class is given in \cite{HS19}. Moreover, it is shown in 
\cite{HS19} that the special case of Theorem \ref{thm:all-k} when $d=2$ and $k=1$
can be proved directly using these modified Henneberg moves.

A more complex Henneberg-type characterization of generic coordinated rigidity when $d=2$ and $k=2$ is also established in \cite{HS19}.
However, in contrast to the $d=2$ and $k=1$ case, the proof relies on Theorem \ref{thm:all-k}. We refer the reader to \cite{HS19} for further details and conjectures.

\subsection{Coordination classes maintaining sums or ratios of edge lengths}
It is natural to try to extend our work to coordinated frameworks in which the edges
in each class have to coordinate their edge lengths changes in a different fashion, say by maintaining
the sum of the edge lengths or the pairwise ratios of the edge lengths. 
The rigidity analysis of such coordinated frameworks seems more complex than the one considered in this paper.

For example, in dimension $2$ it is well known that the standard rigidity matroid is equivalent to the parallel drawing matroid (see \cite{Wh96}, for example), and it is easy to see that if the lengths of two edges of $(G,p)$ are allowed to change so that the ratio of their lengths is maintained, then in the corresponding parallel drawing matroid these edges are allowed to change their direction, but their angle is maintained. The rigidity analysis of frameworks with angle constraints, however, is known to be very difficult in general.

\subsection{Global coordinated rigidity}
This paper deals exclusively with local rigidity of coordinated frameworks.
A stronger condition than local rigidity is \defn{global rigidity}, which 
would, in the coordinated setting, require that \textit{any} $(G,c,q,s)$
that is equivalent to a coordinated framework $(G,c,p,r)$ be congruent to it (as opposed to just the $(G,c,q,s)$ with $(q,s)$ lying in  a small neighborhood of $(p,r)$).  See Figure \ref{fig:equiv} for an example of a coordinated framework in the plane which is locally rigid, but not globally rigid. Whether global rigidity is a generic property for coordinated frameworks is open.

For bar-joint frameworks, that global rigidity is a generic property is 
a deep result of Gortler, Healy and Thurston \cite{GHT}, who built on 
work of Connelly \cite{C82,C05} and Hendrickson \cite{H92}.  It would be 
interesting to know how much of the theory underlying \cite{GHT}
carries over to the coordinated setting.

\section{Acknowledgements}
This work started at the 2016 ICMS rigidity workshop and a 
follow-up meeting at Lancaster University.  
We thank Walter Whiteley for his encouragement to work on these types of problems, as well as Tibor Jord\'an, Simon Guest and Tony Nixon for helpful discussions about matroid unions, statics, and rigidity circuits, respectively. Helpful suggestions from the anonymous referees improved the paper.

\bibliographystyle{abbrvnat}
\bibliography{coord}

\end{document}